\documentclass[12pt,a4paper]{amsart}
\usepackage{amssymb,amsmath}
\usepackage{hyperref}

\numberwithin{equation}{section}
\newtheorem{theorem}{Theorem}[section]
\newtheorem{lemma}[theorem]{Lemma}

\theoremstyle{definition}
\newtheorem{definition}[theorem]{Definition}
\newtheorem{example}[theorem]{Example}
\newtheorem{remark}[theorem]{Remark}

\newtheorem{problem}[theorem]{Problem}

\newcommand\Supp{\operatorname{Supp}}
\newcommand\Ass{\operatorname{Ass}}
\newcommand\Assh{\operatorname{Assh}}
\newcommand\Ann{\operatorname{Ann}}

\newcommand\Tor{\operatorname{Tor}}
\newcommand\Hom{\operatorname{Hom}}

\newcommand\Ext{\operatorname{Ext}}
\newcommand\Rad{\operatorname{Rad}}

\newcommand\inj{\operatorname{inj}}

\newcommand\height{\operatorname{height}}
\newcommand\grade{\operatorname{grade}}

\newcommand{\qism}{\stackrel{\sim}{\longrightarrow}}

\begin{document}
\author[W. Mahmood \and P. Schenzel]{Waqas Mahmood \quad \and \quad Peter Schenzel}
\title[On invariants of local cohomology]{On invariants and endomorphism rings of certain local cohomology modules}

\address{Abdus Salam School of Mathematical Sciences, GCU, Lahore Pakistan}%
\email{ waqassms$@$gmail.com}

\address{Martin-Luther-Universit\"at Halle-Wittenberg,
Institut f\"ur Informatik, D --- 06 099 Halle (Saale),
Germany}
\email{peter.schenzel@informatik.uni-halle.de}

\thanks{This research was partially supported by the Higher Education Commission, Pakistan}


\begin{abstract} Let $(R,\mathfrak{m})$ denote an $n$-dimensional Gorenstein ring. For an ideal $I \subset R$
with $\grade I = c$ we
define new numerical invariants $\tau_{i,j}(I)$ as the socle dimensions
of $H^i_{\mathfrak{m}}(H^{n-j}_I(R))$. In case of a regular local ring
containing a field these numbers coincide with the Lyubeznik numbers
$\lambda_{i,j}(R/I)$. We use $\tau_{d,d}(I), d = \dim R/I,$
to characterize the surjectivity of the natural homomorphism
$f : \hat{R} \to \Hom_{\hat{R}}(H^c_{I\hat{R}}(\hat{R}),H^c_{I\hat{R}}(\hat{R}))$. As a
technical tool we study several natural homomorphisms. Moreover we prove
a few results on $\tau_{i,j}(I)$.
\end{abstract}

\maketitle

\section{Introduction} Let $(R,\mathfrak{m})$ denote a local ring. For
an ideal $I \subset R$ let $H^i_I(R), i\in \mathbb{Z},$ denote the local cohomology modules of $R$ with respect to $I$ (see \cite{Gr} and \cite{BS}
for the definitions). In recent research there is an interest in the
study of the endomorphism rings $\Hom_R(H^i_I(R), H^i_I(R))$ for certain
$i$ and $I$. This was started by Hochster and Huneke (see \cite{HH}) for
$i = \dim R$ and $I = \mathfrak{m}$. See also \cite{MS} for a generalization to an arbitrary ideal $I$ and $i = \dim R$.

The case of a cohomologically complete intersection with $i = \grade I$ was investigated at first by Hellus and St\"uckrad (see \cite{HSt}). They have shown that the natural map
\[
f : \hat{R} \to
\Hom_{\hat{R}}(H^i_{I\hat{R}}(\hat{R}),H^i_{I\hat{R}}(\hat{R}))
\]
is an isomorphism.  In the case of a Gorenstein ring $(R,\mathfrak{m})$ it follows (see \cite{Sch5}) that
$\Hom_{\hat{R}}(H^i_{I\hat{R}}(\hat{R}),H^i_{I\hat{R}}(\hat{R}))$ is a commutative ring. Here we continue
with the investigations of $f$.

\begin{theorem} \label{1.1} Let $(R,\mathfrak{m})$ denote an $n$-dimensional Gorenstein ring.
For an ideal $I \subset R$ with $\grade I = c$ and $d = \dim R/I$ the following conditions are equivalent:
\begin{itemize}
\item[(i)] $\dim_k \Hom_R(k,H^d_{\mathfrak{m}}(H^c_I(R))) = 1.$
\item[(ii)] The natural map $f : \hat{R} \to
\Hom_{\hat{R}}(H^c_{I\hat{R}}(\hat{R}),H^c_{I\hat{R}}(\hat{R}))$ is surjective.
\end{itemize}
Let $S = \cap (R\setminus \mathfrak{p})$ where the intersection is taken over all associated
prime ideal $\mathfrak{p} \in \Ass_R R/I$ with $\dim R/\mathfrak{p} = \dim R/I$. Then
$\ker f = 0R_S \cap R$. That is, if $\hat{R}$ is a domain then $\Ann H^c_{I\hat{R}}(\hat{R}) = 0$.
\end{theorem}

As a consequence of Theorem \ref{1.1} there is a numerical condition for $f$ being an isomorphism.
In general the socle dimension of $H^d_{\mathfrak{m}}(H^c_I(R))$ is not finite. In case of $(R,\mathfrak{m})$
a regular local ring containing a field it is known that $\dim_k \Hom_R(k, H^d_{\mathfrak{m}}(H^c(R)))$
is equal to the Lyubeznik number $\dim_k \Ext_R^d(k, H^c_I(R))$. Another part of these investigations
is the relation of the Bass-Lyubeznik number to the socle dimension of $H^d_{\mathfrak{m}}(H^c_I(R))$. To
this end we consider several re-interpretations of the natural map$f$ by Local and Matlis Duality resp.
This is related to the problem whether the natural homomorphism $\Ext_R^d(k, H^c_I(R)) \to k$ is non-zero, a
question originally posed by the second author and Hellus (see \cite{HeS}). Under the additional
assumption of $\inj\dim_R H^c_I(R) \leq d$ we are able to prove that if $\Ext_R^d(k,H^c_I(R)) \to k$
is non-zero (resp. isomorphic), then $f$ is injective (resp. isomorphic). Note that if $(R,\mathfrak{m})$ is a regular local ring
containing a field $\inj\dim_R H^c_I(R) \leq d$ holds by the results of Huneke and Sharp resp. Lyubeznik
(see \cite{HS} resp. \cite{Ly}). It is unknown to us whether this is true for any regular local ring.
In the final section we discuss the $\tau$-numbers in more detail.

\section{Preliminaries}
Let $(R,\mathfrak m)$ denote a (local) Gorenstein ring of dimension $n$.
Let $I\subset R$ be an ideal of $\grade(I) = c.$ Then we have that $\height(I)= c$ and $d := \dim_R(R/I) = n-c$. For the definition and basic
results on local cohomology theory we refer to \cite{Gr} and the textbook
\cite{BS}. In the following we need a slight extension of the Local Duality Theorem as its was proved at first by Grothendieck (see \cite{Gr}).

For an arbitrary local ring $(R,\mathfrak{m})$ let $\hat{R}$ denote the
$\mathfrak{m}$-adic completion. Moreover $E = E_R(k), k = R/\mathfrak{m}$,
denotes the injective hull of the residue field $k$.

Let $\underline{x}= x_1, \ldots ,x_r\in I$ denote a system of elements of $R$ such that $\Rad I= \Rad(\underline{x})R$. We consider the \v{C}ech complex $\Check{C}_{\underline{x}}$ with respect to $\underline{x}= x_1,...,x_r$ (see \cite{BS} and also \cite{Sch1} for its definition). Then
\[
H^i(\Check{C}_{\underline{x}}\otimes_R M)\simeq H^i_I(M)
\]
for $i\in \mathbb Z$ and an $R$-module $M$, where $H^i_I(M), i \in \mathbb{Z},$ denote the local cohomology of $M$ with support in $I$.

\begin{lemma} \label{2.1}
Let $(R,\mathfrak m)$ denote a  Gorenstein ring. Let $M$ denote an
arbitrary $R$-module. Then there are the following natural isomorphisms
\begin{itemize}
\item[(a)]
$\Tor_{n-i}^R(M, E) \simeq H^i_{\mathfrak{m}}(M)$ and
\item[(b)]
$\Hom_R(H^i_{\mathfrak{m}}(M), E) \simeq \Ext^{n-i}_R(M, {\hat{R}})$
\end{itemize}
for all $i\in \mathbb Z$.
\end{lemma}

\begin{proof}
Let $\underline{x}= x_1, \ldots ,x_n$ denote a system of parameters of $R$. Then $H^i_{\mathfrak m}(R) \simeq H^i(\Check{C}_{\underline{x}})$ and therefore
\[
H^i_{\mathfrak m}(R) \simeq H^i(\Check{C}_{\underline{x}}) = 0
\]
for all $i \not= n$ and $H^n_{\mathfrak m}(R)\simeq E$. Note that $R$ is
a Gorenstein ring. That is, $\Check{C}_{\underline{x}}[-n]$ provides
a flat resolution of $E.$ Whence by the definitions
\[
\Tor_{n-i}^R(M, E) \simeq H^i(\Check{C}_{\underline{x}} \otimes_R M) \simeq H^i_{\mathfrak m}(M)
\]
for all $i \in \mathbb Z$ and any $R$-module $M.$ This proves the first statement. By applying the Matlis duality functor it follows that
\[
\Hom_R(\Tor_{n-i}^R(M, E), E)\simeq \Hom_R(H^i_{\mathfrak m} (M), E)
\]
Therefore the duality beween "Ext" and "Tor"  yields the
isomorphism
\[
\Hom_R(\Tor_{n-i}^R(M, E), E)\simeq \Ext^{n-i}_R(M,\Hom_R(E, E)).
\]
By Matlis duality we get $\Hom_R(E, E)\simeq {\hat{R}}$, which completes the proof.
\end{proof}

Note that the proof of Lemma \ref{2.1} is well-known (see for example \cite{Hu}). The new aspect - we need in the following - is its validity for an arbitrary $R$-module $M$.

Now let us recall the definition of the truncation complex as it was
introduced in \cite{Sch2}. Assume that $(R, \mathfrak{m})$ is an $n$-dimensional Gorenstein ring.
Let $E^{\cdot}_R(R)$ denote the minimal injective resolution of $R.$ Let $\Gamma_I(-)$ denote the section functor with support in $I$. Then $\Gamma_I(E^{\cdot}_R(R))^i= 0$ for all $i< c$ (see \cite{Sch2} for more
details).

\begin{definition} \label{2.2}
Let $C^{\cdot}_R(I)$ be the cokernel of the embedding
$H^c_I(R)[-c]\to \Gamma_I(E^{\cdot}_R(R))$ considered as a morphism of complexes. It is called the
truncation complex of $R$ with respect to $I$. Then there ia a short exact
sequence of complexes
\[
0\to H^c_I(R)[-c]\to \Gamma_I(E^{\cdot}_R(R))\to C^{\cdot}_R(I)\to 0.
\]
In particular  $H^i(C^{\cdot}_R(I))= 0$ for all $i\leq c$ or $i> n$ and $H^i(C^{\cdot}_R(I))\simeq H^i_I(R)$ for all $c< i\leq n.$
\end{definition}

The advantage of the truncation complex is the separation of the
properties of the cohomology module $H^c_I(R)$ of those of
the other cohomology modules $H^i_I(R)$ for $i\neq c$. The truncation
complex is helpful in order to construct a few natural homomorphisms.
But as a first application of the truncation complex we have the following
result.

\begin{lemma} \label{2.3}
Let $(R,\mathfrak m)$ be a local ring. Let $X$ denote an $R$-module such
that $\Supp_R(X)\subseteq V(I)$ for an ideal $I$ of $\grade I= c$. Then
there is a natural isomorphism
\[
\Hom_R(X,H^c_{I} (R)) \simeq \Ext^c_R(X, R)
\]
and $\Ext^i_R(X, R)= 0$ for all $i < c$.
\end{lemma}

\begin{proof} For the proof see  \cite[Theorem 2.3]{Sch3}.
\end{proof}

The truncation complex provides a few natural maps that arise also
as edge homomorphisms of certain spectral sequences.

\begin{lemma} \label{2.4}
With the previous notation and Definition \ref{2.2} there are the following natural homomorphisms
\begin{itemize}
\item[(a)] $H^d_{\mathfrak m}(H^c_I(R))\to E$,
\item[(b)] ${\hat{R}}\to \Ext^{c}_R(H^c_I(R), {\hat{R}})$,
\item[(c)] $\Tor_{c}^R(E,H^c_I(R))\to E$,
\item[(d)] $\Ext^{d}_R(k,H^c_I(R))\to k$, and
\item[(e)] $\Tor_{c}^R(k,H^c_I(R))\to k$.
\end{itemize}
\end{lemma}

\begin{proof}
Let $\underline{x}= x_1, \ldots,x_n\in {\mathfrak m}$ denote a system of
parameters of $R$. Let $\Check{C}_{\underline{x}}$ denote the \v{C}ech complex with respect to
$\underline{x}$. It induces a homomorphism of complexes of $R$-modules
\[
\Check{C}_{\underline{x}}\otimes_{R} H^c_I(R)[-c]\to \Check{C}_{\underline{x}}\otimes_{R} \Gamma_I(E^{\cdot}_R(R))
\]
By view of \cite[Theorem 3.2]{Sch4} we have a morphism of complexes
\[
\Gamma_{\mathfrak m}(\Gamma_I(E^{\cdot}_R(R))) \to \Check{C}_{\underline{x}}\otimes_{R} \Gamma_I(E^{\cdot}_R(R))
\]
that induces an isomorphism in cohomology.
Moreover there is an identity of functors $\Gamma_{\mathfrak m}(\Gamma_I(-))= \Gamma_{\mathfrak m}(-)$. Because $R$ is a Gorenstein ring it follows that $\Gamma_{\mathfrak m}(E^{\cdot}_R(R))\simeq E[-n]$.  Therefore the above
homomorphism considered in degree $n$ induces the homomorphism in $(a)$.

By view of the Local Duality (see \ref{2.1}) we have the isomorphism
\[
H^d_{\mathfrak m}(H^c_I(R))\simeq \Tor_{c}^R(E,H^c_I(R)).
\]
Therefore the homomorphism in $(a)$ implies the homomorphism in $(c)$.

Again by Local Duality (see \ref{2.1}) there is the isomorphism
\[
\Ext^{c}_R(H^c_I(R), {\hat{R}})\simeq \Hom_R(H^d_{\mathfrak m}(H^c_I(R)), E)
\]
By the Matlis Duality it follows that the homomorphism in $(a)$ provides the homomorphism in $(b)$.

Now let $F_{\cdot}$ denote a free resolution of the residue field $k$.
Then the truncation complex induces -- by tensoring with $F_{\cdot}$ --
the following morphism of complexes of $R$-modules
\[
(F_{\cdot}\otimes_{R} H^c_I(R))[-c]\to
F_{\cdot}\otimes_{R}\Gamma_I(E^{\cdot}_R(R)).
\]
Now let $\underline{y} = y_1,\ldots, y_r$ a generating set for the ideal $I$.
Let $\Check{C}_{\underline{y}}$ denote the \v{C}ech complex with respect
to $\underline{y}$. Because of the natural morphisms  $\Gamma_I(E^{\cdot}_R(R)) \to
\Check{C}_{\underline{y}}\otimes_{R} E^{\cdot}_R(R)$ (see \cite{Sch4}) induces an isomorphism in cohomology.
We get the quasi-isomorphism of complexes
\[
F_{\cdot}\otimes_{R} \Gamma_I(E^{\cdot}_R(R))\qism F_{\cdot}\otimes_{R} \Check{C}_{\underline{y}}\otimes_{R} E^{\cdot}_R(R)
\]
But now we have a quasi-isomorphism $F_{\cdot}\otimes_{R} \Check{C}_{\underline{y}} \qism F_{\cdot}$. Then the homology in degree $0$ induces the homomorphism in $(e)$. A similar argument with $\Hom_R(F_{\cdot}, -)$ implies the homomorphism in $(d)$. This completes the proof.
\end{proof}

\section{Natural homomorphisms}
Let $M$ denote an $R$-module. Then there is the natural homomorphism
\[
\Phi: R\to \Hom_R(M,M), r\to f_r
\]
where $f_r: M\to M$, $m\to rm$, for all $m\in M$ and $r\in R.$ This homomorphism of $R$ into the endomorphism ring is in general neither injective nor surjective. Of course the endomorphism ring is in general not a commutative ring.

\begin{remark} \label{3.1}
(A) For a local Gorenstein ring $(R, \mathfrak m)$ we use the above notations and conventions. By view of Lemma \ref{2.4} there are the following natural homomorphisms
\begin{itemize}
\item[]
$\varphi_1: H^d_{\mathfrak m}(H^c_I(R))\to E$,
\item[]
$\varphi_2: \Tor_{c}^R(E,H^c_I(R))\to E$,
\item[]
$\varphi_3: {\hat{R}}\to \Ext^{c}_R(H^c_I(R), {\hat{R}})$, and
\item[]
$\varphi_4: {\hat{R}}\to \Hom_{\hat{R}}(H^c_{I{\hat{R}}}({\hat{R}}), H^c_{I{\hat{R}}}({\hat{R}}))$.
\end{itemize}

(B) It is known (see \cite[Theorem 3.2 (a)]{Sch5}) that $S =
\Hom_R(H^c_I(R),H^c_I(R))$ is a commutative ring for an ideal $I$
of a Gorenstein ring $(R,\mathfrak{m})$ and $c = \grade I$.
\end{remark}

In the following it will be our intention to investigate the homomorphisms
of Remark \ref{3.1} in more detail. In particular we have the following
result:

\begin{theorem} \label{3.2} Let $I \subset R$ denote an ideal of height
$c$ of a Gorenstein ring $(R,\mathfrak{m})$. With the previous notation
we have:
\begin{itemize}
\item[(a)] $\varphi_i$, $i= 1,\ldots,4$, are all non-zero homomorphisms.
\item[(b)] The following conditions are equivalent:
\begin{itemize}
\item[(i)] $\varphi_1$ is injective (resp. surjective).
\item[(ii)] $\varphi_2$ is injective (resp. surjective).
\item[(iii)] $\varphi_3$ is surjective (resp. injective).
\item[(iv)] $\varphi_4$ is surjective (resp. injective).
\end{itemize}
\item[(c)] The natural map $\psi : \Hom_R(k,H^d_{\mathfrak{m}}(H^c_I(R))) \to k$ is non-zero and therefore onto.
\end{itemize}
\end{theorem}

\begin{proof}
First of all we prove that there is a natural isomorphism
\[
\Ext^{c}_R(H^c_I(R), {\hat{R}})\simeq \Hom_{\hat{R}}(H^c_{I{\hat{R}}}({\hat{R}}), H^c_{I{\hat{R}}}({\hat{R}}))
\]
This follows from Lemma \ref{2.3} since
\[
\Ext^{c}_R(H^c_I(R), {\hat{R}})\simeq \Ext^{c}_{\hat{R}}(H^c_{I{\hat{R}}}(\hat{R}), {\hat{R}})
\]
and ${\hat{R}}$ is a flat $R$-module. Moreover we have $c= \grade I{\hat{R}}.$

Because $\varphi_4$ is non-zero and the previous isomorphism is natural we get that $\varphi_3$ is non-zero too. By Lemma \ref{2.1} the
module $\Ext^{c}_R(H^c_I(R), {\hat{R}})$ is the Matlis dual
of $H^d_{\mathfrak{m}}(H^c_I(R))$. So $\varphi_1$ is also a non-zero
homomorphism. By Local Duality (see Lemma \ref{2.1}) the last
module is isomorphic to $\Tor_c^R(H^c_I(R),E)$. This finally
proves also that $\varphi_2 \not= 0$. Therefore the statement in (a) is shown to be true.
With the same arguments the equivalence of the conditions in the
statement $(b)$ follows by Matlis Duality.

For the proof of (c) first note that $\varphi_4$ is non-zero and $1 \mapsto
\operatorname{id}$. So the reduction modulo $\mathfrak{m}$ induces a non-zero
homomorphism
$
 k \to k \otimes \Hom_{\hat{R}}(H^c_{I{\hat{R}}}({\hat{R}}), H^c_{I{\hat{R}}}({\hat{R}})).
$
But this is the Matlis dual of $\psi$ as it follows by Local Duality.
\end{proof}

In order to investigate the natural homomorphisms (d), (e) of Lemma \ref{2.4}
we need a few more preliminaries.
First note that by Lemma \ref{2.4} (see also \cite[Section 6]{Sch5}) there is a commutative diagram
\[
\begin{array}{ccc}
  \Ext^d_R(k, H^c_I(R)) &  \to & k \\
  \downarrow &   & \downarrow \\
   H^d_{\mathfrak m}(H^c_I(R)) & \to & E
\end{array}
\]
here the right vertical morphism is - by construction - the natural inclusion and the left vertical morphism is the direct limit of the natural homomorphisms
\[
\lambda_\alpha: \Ext^d_R(k, K(R/I^\alpha))\to H^d_{\mathfrak m}(K(R/I^\alpha))
\]
for all $\alpha\in \mathbb{N}$, where $K(M)$ denote the canonical module of $M$. Then it induces the following commutative diagram
\[
\begin{array}{ccc}
  \Ext^d_R(k, H^c_I(R)) &  \stackrel{\varphi}{\to} & k \\
  \lambda \downarrow &   & \| \\
   \Hom_R(k, H^d_{\mathfrak m}(H^c_I(R))) & \stackrel{\psi}{\to} & k
\end{array}
\]
It is of some interest to investigate the homomorphism $\varphi$.
In fact, it was conjectured that $\varphi$ is in general non-zero. This
was shown to be true for $R$ a regular local ring containing a field (see
\cite[Theorem 1.3]{Sch5}) while $\psi$ is non-zero in a Gorenstein ring
as shown above (see Theorem \ref{3.2} (c)).

For our purposes here let us summarize the main results of
Huneke and Sharp resp. Lyubeznik (see \cite{HS} and \cite{Ly}).

\begin{theorem} \label{3.5}
Let $(R, \mathfrak m)$ be a regular local ring containing a field. Let
$I\subset R$ be an ideal of grade $c$. Then for all $i,j \in \mathbb Z$
the following results are true:
\begin{itemize}
\item[(a)]  $H^j_{\mathfrak m}(H^{i}_I(R))$ is an injective $R$-module.
\item[(b)] $\inj \dim_R(H^{i}_I(R))\leq \dim_R(H^{i}_I(R))\leq \dim R-i.$
\item[(c)] $\Ext^j_R(k, H^i_I(R))\simeq \Hom_R(k, H^j_{\mathfrak m}(H^{i}_I(R)))$
and its dimension is finite.
\end{itemize}
\end{theorem}

\begin{proof} For the proof we refer to \cite{HS} and \cite{Ly}.
\end{proof}

In particular the result of Theorem \ref{3.5} (c) is a motivation for the following Definition.

\begin{definition} \label{3.3}
If $I\subset R$ is an ideal of grade $c$ in $(R, \mathfrak m)$ a Gorenstein ring. Then
\[
\tau_{i,j}(I):= \dim_k \Hom_R(k, H^i_{\mathfrak m}(H^{n-j}_I(R)))
\]
is called $\tau$-number of type $(i,j)$ of $I$.
\end{definition}

\begin{remark} \label{3.4}
Note that $\tau_{i,j}(I)$ is the socle dimension of the local
cohomology module $H^i_{\mathfrak{m}}(H^{n-j}_I(R))$. In the case
of a regular local ring $(R,\mathfrak{m})$ containing a field it was
shown by Huneke and Sharp resp. Lyubeznik (see Theorem \ref{3.5})
that
\[
\Ext_R^i(k,H^{n-j}_I(R)) \simeq \Hom_R(k,H^i_{\mathfrak{m}}(H^{n-j}_I(R)))
\]
for all $i, j \in \mathbb{Z}$. These Bass numbers are also called
Lyubeznik numbers. In fact, Lyubeznik showed that they are invariants of $R/I$
and finite.

By the example of Hartshorne (see Example \ref{5.3}) it turns out that for a Gorenstein $(R,
\mathfrak{m})$ the integers $\tau_{i,j}(I)$ need not to be finite.
\end{remark}

As a first result on the integers $\tau_{i,j}(I)$ let us consider its vanishing resp. non-vanishing.

\begin{theorem} \label{3.6} Let $I \subset R$ denote an ideal of the local
Gorenstein ring $R$ and $c = \grade I$.
\begin{itemize}
\item[(a)] $\tau_{d,d}(I) \not= 0$, where $d = \dim R/I$.
\item[(b)] $\tau_{i,j}(I) = 0$ if and only if $H^i_{\mathfrak{m}} (H^{n-j}_I(R)) = 0$, where $n = \dim R$.
\item[(c)] $\tau_{i,j}(I) < \infty$ if and only if $H^i_{\mathfrak{m}} (H^{n-j}_I(R))$ is an Artinian
$R$-module.
\end{itemize}
\end{theorem}

\begin{proof}  The statement in (a) is a consequence of Theorem \ref{3.2} (c). For the proof of (b)
first note that any element of $H^i_{\mathfrak{m}} (H^{n-j}_I(R))$ is annihilated by a power of
$\mathfrak{m}$. Then the statement is true by the well-known fact that the module is zero if and only if
its socle vanishes.

In order to prove (c) recall again that $\Supp H^i_{\mathfrak{m}} (H^{n-j}_I(R)) \subseteq V(\mathfrak{m})$. Then the
finiteness of the socle dimension is equivalent to the Artinianness of a module.
\end{proof}

\section{The endomorphism ring}
As above let $I \subset R$ denote an ideal of grade $c$ in the
Gorenstein ring $(R,\mathfrak{m})$. In this section we shall
investigate the natural homomorphism
\[
\varphi_4: \hat{R} \to
\Hom_{\hat{R}}(H^c_{I{\hat{R}}}({\hat{R}}), H^c_{I{\hat{R}}}({\hat{R}})).
\]
As a first step we want to characterize when it will be surjective.
To this end we need a few preparations. For an ideal $I$ let
\[
\Assh R/I = \{\mathfrak{p} \in \Ass R/I | \dim R/\mathfrak{p}
= \dim R/I \}.
\]
Then we define the multiplicatively closed set $S = \cap (R\setminus \mathfrak{p}$),
where the intersection is taken over all $\mathfrak{p}
\in \Assh R/I$. For $n \in \mathbb{N}$ the $n$-th symbolic power
$I^{(n)}$ of $I$ is defined as $I^{(n)} = I^nR_S \cap R$.

\begin{definition} \label{4.1} With the previous notation put
\[
u(I) = \cap_{n \geq 1} I^{(n)} = 0R_S \cap R.
\]
Note that $u(I)$ is equal to the intersection of all $\mathfrak{p}$-primary
components $\mathfrak{q}$ of a minimal primary decomposition of the zero ideal
$0$ in $R$ satisfying $S \cap \mathfrak{p} = \emptyset$ .
\end{definition}

\begin{lemma}\label{4.2}
Let $I\subset R$ is an ideal of grade $c$ in $(R, \mathfrak m)$
a Gorenstein ring. Then $\ker\varphi_4 = u(I\hat{R})$ and
$\Ann_{\hat{R}} H^c_{I\hat{R}}(\hat{R}) = u(I\hat{R})$.
\end{lemma}

\begin{proof} It is known (see \cite[Theorem 3.2]{Sch5}) that
$\ker \varphi_4 = u(I\hat{R})$. Then cleary $u(I\hat{R}) = \Ann_{\hat{R}} H^c_{I\hat{R}}(\hat{R})$.
\end{proof}

It follows that $H^c_{I\hat{R}}(\hat{R})$ is a torsion-free $\hat{R}$-module
if $\hat{R}$ is a domain. In the next we investigate when $\varphi_4$ is
surjective.

\begin{theorem} \label{4.3}
Let $(R, \mathfrak m)$ denote an $n$-dimensional Gorenstein ring. let $I\subset R$
be an ideal and $\grade(I)= c$. Set $d:= n-c$ then the following conditions are equivalent:
\begin{itemize}
\item[(i)] $\tau_{d,d}(I)= 1.$
\item[(ii)] The natural homomorphism $
\Hom_R(k, H^d_{\mathfrak m}(H^c_I(R)))\to k
$
is an isomorphism.
\item[(iii)] The natural homomorphism
$
{\hat{R}}\to \Hom_{\hat{R}}(H^c_{I{\hat{R}}}({\hat{R}}), H^c_{I{\hat{R}}}({\hat{R}}))
$
is surjective.
\end{itemize}
\end{theorem}

\begin{proof}
First of all note that we may assume that $R$ is complete without loss of generality.

(iii)$\Rightarrow$(ii): Fix the above notation. If $\varphi_4$ is surjective
it follows (see Theorem \ref{3.2}) that
\[
\varphi_1 : H^d_{\mathfrak m}(H^c_I(R))\to E
\]
is injective. Since $\psi$ is non-zero it provides (by applying $\Hom_R(k,-)$)
that
\[
\psi: \Hom_R(k, H^d_{\mathfrak m}(H^c_I(R)))\to k
\]
is an isomorphism.

(ii)$\Rightarrow$(i): This is obviously true by the definition of $\tau_{d,d}$.

(i)$\Rightarrow$(iii): In order to simplify notation we put $S= \Hom_R(H^c_I(R), H^c_I(R))$.
Then we have to show that the natural homomorphism $R\to S$ is surjective.
Note that $\varphi_4$ is non-zero since $1\to \operatorname{id}$ and $H^c_I(R) \not= 0$. Moreover
\[
S\simeq \Hom_R(H^d_{\mathfrak m}(H^c_I(R)), E)
\]
as follows because of $S\simeq \Ext^c_R(H^c_I(R), R)$ (see Lemma \ref{2.3}) and by
the Local Duality Theorem. Therefore we have isomorphisms
\[
S/\mathfrak mS\simeq k\otimes_R \Hom_R(H^d_{\mathfrak m}(H^c_I(R)), E)\simeq \Hom_R(\Hom_R(k,H^d_{\mathfrak m}(H^c_I(R))), E)
\]
By the assumption and Matlis Duality this implies that $\dim_k S/\mathfrak mS = 1$.

Now we claim that $S$ is $\mathfrak m$-adically complete. With the same duality argument as above there are isomorphisms
\[
S/\mathfrak m^\alpha S\simeq \Hom_R(\Hom_R(R/\mathfrak m^\alpha R,H^d_{\mathfrak m}(H^c_I(R))), E).
\]
for all $\alpha\in \mathbb N$. Now both sides form an inverse system of modules. Since inverse limits commutes with direct limits into $\Hom_R(-,-)$ at the first place. It provides an isomorphism
\[
\varprojlim S/\mathfrak m^\alpha S\simeq \Hom_R(H^0_{\mathfrak m}(H^d_{\mathfrak m}(H^c_I(R))), E)
\]
as follows by passing to the inverse limit. But $H^0_{\mathfrak m}(H^d_{\mathfrak m}(H^c_I(R)))\simeq
H^d_{\mathfrak m}(H^c_I(R))$ since $H^d_{\mathfrak m}(H^c_I(R))$ is an $R$-module whose
support is contained in $V(\mathfrak m).$
Finally this provides an isomorphism
\[
\lim_{\longleftarrow} S/\mathfrak m^\alpha S\simeq S
\]
and hence $S$ is $\mathfrak m$-adically complete.

By virtue of \cite[Theorem 8.4]{M} $S$ is a finitely generated $R$-module, in fact a cyclic
module $S \simeq R/J$. This proves the statement in (iii).
\end{proof}

As shown above there is the following commutative diagram
\[
\begin{array}{ccc}
  \Ext^d_R(k, H^c_I(R)) &  \stackrel{\varphi}{\to} & k \\
  \lambda \downarrow &   & \| \\
   \Hom_R(k, H^d_{\mathfrak m}(H^c_I(R))) & \stackrel{\psi}{\to} & k.
\end{array}
\]
Whence there is a relation between the Bass number $\mu_d(\mathfrak m, H^c_I(R))$ and $\tau_{d,d}(I)$.
In fact, if $(R, \mathfrak m)$ is regular local ring containing a field, then $\lambda$ is an
isomorphism (see \cite{Sch5}).
It is conjectured (\cite{HeS}) that $\varphi$ is non-zero.
Note that if $\varphi$ is non-zero then $\psi$ is non-zero too and both maps are surjective.
In the next we want to continue with the investigations of $\varphi$.

\begin{theorem} \label{4.4}
Let $(R, \mathfrak m)$ be Gorenstein ring. Suppose that $I\subset R$ is an ideal of
grade $c$ such that $\inj \dim_R(H^c_I(R))\leq d = n-c$.
Suppose that the natural homomorphism
\[
\varphi: \Ext^d_R(k, H^c_I(R))\to k
\]
is non-zero (resp. isomorphic). Then
\[
\varphi_4: {\hat{R}}\to \Hom_{\hat{R}}(H^c_{I{\hat{R}}}({\hat{R}}), H^c_{I{\hat{R}}}({\hat{R}}))
\]
is injective (resp. isomorphic).
\end{theorem}

\begin{proof}
Since $\hat{R}$ is a flat $R$-module, the injective dimension of $H^c_I(R)$
is finite if and only if the injective dimension of $H^c_{I\hat{R}}(\hat{R})$
is finite. Moreover $\varphi$ is non-zero (resp. isomorphic) if and only
if $\Ext^d_{\hat{R}}(k,H^c_{I\hat{R}}(\hat{R}) )\to k$ is
non-zero (resp. isomorphic). So without loss of generality we may assume that
$R$ is complete.

It is a consequence of the truncation complex that there are natural homomorphisms
\[
\varphi_\alpha: \Ext^d_R(R/\mathfrak m^\alpha, H^c_I(R))\to \Hom_R(R/\mathfrak m^\alpha, E)
\]
for all $\alpha\in \mathbb N$.

Put $H = H^c_I(R)$. Then the short exact sequence
\[
0\to \mathfrak m^\alpha/\mathfrak m^{\alpha+1}\to R/\mathfrak m^{\alpha+1}\to R/\mathfrak m^{\alpha}\to 0
\]
induces a commutative diagram with exact rows
\[
\begin{array}{cccccccc}
   &  & \Ext^d_R(R/\mathfrak m^\alpha, H) & \to & \Ext^d_R(R/\mathfrak m^{\alpha+1}, H) & \to & \Ext^d_R(\mathfrak m^{\alpha}/\mathfrak m^{\alpha+1}, H) & \to 0\\
    &   & \downarrow {\varphi_\alpha }&  & \downarrow {\varphi_{\alpha+1}} &   & \downarrow f  &\\
 0 & \to & \Hom_R(R/\mathfrak m^\alpha, E) & \to & \Hom_R(R/\mathfrak m^{\alpha+1}, E) & \to & \Hom_R(\mathfrak m^\alpha/\mathfrak m^{\alpha+1}, E) &\to 0
\end{array}
\]
Recall that the homomorphism $\Ext^d_R(R/\mathfrak m^{\alpha+1}, H) \to
\Ext^d_R(\mathfrak m^{\alpha}/\mathfrak m^{\alpha+1}, H)$ is onto since $\inj \dim
H \leq d.$

Now suppose that $\varphi$ is non-zero, that is, it is surjective. So assume that
$\varphi$ is surjective (resp. isomorphic). Then the natural homomorphism $f$
is surjective (resp. isomorphic). The snake lemma provides by induction on
$\alpha \in \mathbb{N}$ that $\varphi_{\alpha}$ is surjective
(resp.  isomorphic). By passing to the direct limit of the direct systems it induces
the homomorphism
\[
\varphi_1 : H^d_{\mathfrak m}(H^c_I(R))\to E
\]
is surjective (resp. isomorphic). Note that the direct limit is exact on direct systems.
So the statement follows by virtue of Theorem
\ref{3.2}
\end{proof}

\begin{remark} \label{4.5}
(A) The assumption $\inj \dim H^c_I(R) \leq d$ implies that $\inj \dim H^c_I(R) = d$.
This follows since $d = \dim H^c_I(R) \leq \inj \dim H^c_I(R)$ as it is easily
seen.

(B) In the case of a regular local ring $(R,\mathfrak{m})$ containing a field
it is shown (see Theorem \ref{3.5}) that $\inj \dim H^c_I(R) \leq d$. We do not know
whether this holds in any regular local ring.
\end{remark}

\section{The \texorpdfstring{$\tau$}{tau}-numbers}
In the following we shall give two applications concerning the $\tau$-numbers of
the ideal $I \subset R$ of a Gorenstein ring $(R,\mathfrak{m})$. To this end
let us recall the following result:

\begin{lemma} \label{5.1} Let $(R,\mathfrak{m})$ denote an $n$-dimensional Gorenstein
ring. Let $I \subset R$ denote an ideal with $c = \grade I$. Let $C_R^{\cdot}(I)$ denote
the truncation complex of $I$ (see Definition \ref{2.2}). Then there is a short exact
sequence
\[
0 \to H^{n-1}_{\mathfrak m}(C^{\cdot}_R(I)) \to H^d_{\mathfrak m}(H^c_I(R)) \to E \to H^n_{\mathfrak m}(C^{\cdot}_R(I)) \to 0,
 \]
isomorphisms $H^{i-c}_{\mathfrak m}(H^c_I(R)) \simeq H^{i-1}_{\mathfrak m}(C^{\cdot}_R(I))$ for $i < n$ and the vanishing $H^{i-c}_{\mathfrak m}(H^c_I(R)) = 0$ for $i > n.$
\end{lemma}

\begin{proof} For the proof we refer to \cite[Lemma 2.2]{HeS}. It is easily seen a
consequence of the truncation complex.
\end{proof}

\begin{theorem} \label{5.2} Let $I \subset R$ denote a 1-dimensional ideal in the
Gorenstein ring $(R,\mathfrak{m})$.
\begin{itemize}
\item[(a)] $\tau_{i,j}(I) = 0$ for all $(i,j) \not\in \{(0,0), (1,1)\}$.
\item[(b)] $\tau_{0,0}(I) = \dim_k \Ext_R^1(k,H^1_{\mathfrak{m}}(H^{n-1}_I(R))) < \infty$ and
$\tau_{1,1}(I) = 1$.
\item[(c)] $\Hom_{\hat{R}}(H^c_{I{\hat{R}}}({\hat{R}}), H^c_{I{\hat{R}}}({\hat{R}})) \simeq \hat{R}/u(I\hat{R})$.
\end{itemize}
\end{theorem}

\begin{proof} Since $H^i(C_R^{\cdot}(I)) \simeq H^i_I(R)$ for all $c < i \leq n$ and zero
otherwise there is a naturally defined morphism $C_R^{\cdot}(I) \to H^n_I(R)[-n]$ that induces
an isomorphism in cohomology. Therefore $H^i_{\mathfrak{m}}(C_R^{\cdot}(I)) \simeq H^{i-n}_{\mathfrak{m}}(H^n_I(R))$
for all $i \in \mathbb{Z}$. Whence there are an exact sequence
\[
0 \to H^{-1}_{\mathfrak{m}}(H^n_I(R)) \to H^1_{\mathfrak{m}}(H^{n-1}_I(R)) \to E \to H^0_{\mathfrak{m}}(H^n_I(R))
\to 0 \]
and isomorphisms $H^{i-n}_{\mathfrak{m}}(H^n_I(R)) \simeq H^{i+2-n}_{\mathfrak{m}}(H^{n-1}_I(R))$ for all $i \not= n, n-1$
as it is a consequence of Lemma \ref{5.1}. This provides all the vanishing of the $\tau$-numbers as claimed
in the statement. To this end recall that $\dim H^i_I(R) \leq n-i$ for $i = n-1, n$.
Moreover it implies a short exact sequence
\[
0 \to H^1_{\mathfrak{m}}(H^{n-1}_I(R)) \to E \to H^0_{\mathfrak{m}}(H^n_I(R)) \to 0.
\]
By applying $\Ext_R^{\cdot}(k, -)$ it induces
isomorphisms
\[
\Hom_R(k,H^1_{\mathfrak{m}}(H^{n-1}_I(R))) \simeq k \text{ and }
\Hom_R(k,H^0_{\mathfrak{m}}(H^n_I(R))) \simeq \Ext_R^1(k,H^1_{\mathfrak{m}}(H^{n-1}_I(R))).
\]
To this end recall the the natural homomorphism $\Hom_R(k,H^1_{\mathfrak{m}}(H^{n-1}_I(R))) \to k$
is not zero (see Theorem \ref{3.5} (c)). Because $H^n_I(R)$ is an Artinian $R$-module (see e.g \cite{MS}) we have $H^0_{\mathfrak{m}}(H^n_I(R)) = H^n_I(R)$ and its socle dimension is finite. This proves the statements in (a) and (b). The claim in (c) is a particular case of Theorem \ref{4.3}.
\end{proof}

In his paper (see \cite{Bl}) Blickle has shown a certain duality statement for Lyubeznik
numbers in the case of cohomologically isolated singularities in a regular local ring containing a field (see also \cite[Corollary 5.4]{Sch5}
for a slight extension). Here we prove a corresponding result for the $\tau$-numbers for an ideal in a Gorenstein ring.

\begin{theorem} \label{5.3} Let $I \subset R$ denote an ideal of grade $c$ in the Gorenstein ring
$(R,\mathfrak{m})$. Suppose that $c < n-1$ and $\Supp H^i_I(R) \subseteq V(\mathfrak{m})$ for all $i \not= c$.
Then
\[
\tau_{0,j}(I) = \tau_{d-j+1,d}(I) - \delta_{d-j+1,d},
\] where $d = \dim R/I$.
\end{theorem}

\begin{proof} Under the assumption of $\Supp H^i_I(R) \subseteq V(\mathfrak{m})$ for all $i \not= c$ it follows that
there are isomorphisms $H^i_{\mathfrak{m}}(C^{\cdot}_R(I)) \simeq H^i(C^{\cdot}_R(I))$ for all $i \in
\mathbb{Z}$. The exact sequence and the isomorphisms shown in Lemma \ref{5.1} provide
the statement of the claim. Note that we have an exact sequence
\[
0 \to H^{n-1}_I(R) \to H^d_{\mathfrak{m}}(H^c_I(R)) \to E.
\]
This in particular provides (as above) that $\tau_{d,d}(I) = \tau_{0,1}(I) + 1$.
\end{proof}

Even under the assumptions of Theorem \ref{5.3} it may happen that $\tau_{d,d}$ is infinite. To this end
consider the following example of Hartshorne (see \cite{Ha}).

\begin{example} \label{5.4} (cf. \cite[Section 2]{Ha} and \cite[Example 3.5]{Sch5})
Let $k$ be a field and let
$A = k[|u,v,x,y|]$ be the formal power series ring in four
variables. Consider the complete Gorenstein ring $R = A/fA,$
where $f = xv-yu.$ Let $I = (x,y)R.$ Then $n = 3, d= 2, c= 1$ and
$\Supp H^2_I(R) = V(\mathfrak{m}), H^3_I(R) = 0$. It was shown by
Hartshorne (see \cite[Section 2]{Ha}) that
\[
\dim_k \Hom_R(k, H^2_I(R)) = \tau_{0,1}(I) = \tau_{2,2}(I) -1
\]
is not finite.
\end{example}

\begin{problem}
It would be of some interest to get a criterion for $\tau_{d,d}(I)$ to be finite
for an ideal $I \subset R$ with $d = \dim R/I$. By Theorem \ref{3.6} this is equivalent to
the Artinianness of $H^d_{\mathfrak{m}}(H^c_I(R))$.
\end{problem}

\noindent\textbf{Acknowledgement.} The authors are grateful to the reviewer for suggestions to
improve the manuscript.


\begin{thebibliography}{9999}
\bibitem[l]{Bl} {\sc M.~Blickle:}  Lyubeznik's invariants for
cohomologically isolated singularities, J. Algebra 308 (2007),
118-123.

\bibitem [2]{BS} {\sc M. Brodmann, R. Sharp:} Local Cohomology. An Algebraic
Introduction with Geometric Applications. Cambr. Stud. in Advanced Math., No. 60. Cambridge University Press, (1998).

\bibitem[3]{MS} {\sc M. Eghbali, P. Schenzel:} On an endomorphism ring of local cohomology, Comm. Algebra, to
appear.

\bibitem[4]{Gr} {\sc A. Grothendieck:} Local Cohomology, Notes by R. Hartshorne, Lecture Notes in
Math. vol. 41, Springer, 1967.

\bibitem[5]{Ha} {\sc R.~Hartshorne:}  Affine duality and
cofiniteness, Invent. Math.  9, 145-164 (1970).


\bibitem[6]{HeS} {\sc M. Hellus, P. Schenzel:} On cohomologically complete intersections,
J. Algebra 320 (2008) 3733-3748.

\bibitem[7]{HSt} {\sc M. Hellus, J. St\"uckrad:} On endomorphism rings of local cohomology modules.
Proc. Am. Math. Soc. 136 (2008) 2333-2341.

\bibitem[8]{HH} {\sc M.~Hochster, C.~Huneke:}  Indecomposable canonical
modules and connectedness, In: Commutative Algebra: Syzygies, Multiplicities, and Birational Algebra,  (Eds.:
W. Heinzer, C. Huneke, J. Sally), Contemporary Math. 159, pp.  197-208, Providence,
1994.

\bibitem[9]{Hu} {\sc C. Huneke :} Lectures on local cohomology (with an Appendix by Amelia Taylor), Contemp. Math. 436, pp. 51-100, Providence, 2007.


\bibitem[10]{HS} {\sc C. Huneke, R.Y. Sharp:} Bass numbers of local cohomology modules, Trans. Amer. Math. Soc. 339 (1993) 765779.

\bibitem[11]{Ly} {\sc G. Lyubeznik:} Finiteness properties of local cohomology modules
(an Application of D-modules to Commutative
Algebra), Invent. Math. 113 (1993) 41-55.

\bibitem[12]{M} {\sc H. Matsumura:} Commutative Ring Theory, Cambridge Studies in Advanced Mathematics 8, Cambridge University Press, Cambridge (1986).


\bibitem[13]{Sch1} {\sc P. Schenzel:} On the use of local cohomology in Algebra and Geometry. In: Six Lectures in
Commutative Algebra, Proceed. Summer School on Commutative Algebra at Centre de
Recerca Matem\`{a}tica, (Ed.: J. Elias, J. M. Giral, R. M. Mir\'{o}-Roig, S. Zarzuela), Progr. Math. 166, pp. 241-292, Birkh\"auser, 1998.

\bibitem[14]{Sch4} {\sc P. Schenzel:} Proregular sequences,
local cohomology, and completion, Math. Scand 92 (2003) 161-180.

\bibitem[15]{Sch2} {\sc P.~Schenzel:}  On birational
Macaulayfications and Cohen-Macaulay canonical modules. J.
Algebra 275 (2004), 751-770.

\bibitem[16]{Sch3} {\sc P. Schenzel:} Matlis duals of local
cohomology modules and their endomorphism rings. Arch. Math.
(Basel) 95 (2010) 115-123.

\bibitem[17]{Sch5} {\sc P. Schenzel:} On the structure of
the endomorphism ring of a certain local cohomology module.
J. Algebra 344 (2011), 229-245.

\end{thebibliography}
\end{document}